\documentclass[12pt]{article}

\usepackage{amsmath,amsthm,amsfonts,amssymb, color}

\newtheorem{theorem}{Theorem}[section]
\newtheorem{corollary}[theorem]{Corollary}

\newtheorem{lemma}[theorem]{Lemma}
\newtheorem{definition}[theorem]{Definition}
\newtheorem{remark}[theorem]{Remark}

\def\cD{\mathcal{D}}

\def\cF{\mathcal{F}}

\def\bD{\mathbb{D}}
\def\bN{\mathbb{N}}
\def\bR{\mathbb{R}}

\def\e{\varepsilon}

\begin{document}

\title{Weak convergence and tightness of probability measures in an abstract Skorohod space}

\author{Raluca M. Balan\footnote{Corresponding author. University of Ottawa, Department of Mathematics and Statistics,
STEM Building, 150 Louis-Pasteur Private, Ottawa, ON, K1N 6N5, Canada. E-mail
address: rbalan@uottawa.ca} \footnote{Research supported by a
grant from the Natural Sciences and Engineering Research Council
of Canada.}\and
Becem Saidani\footnote{University of Ottawa, Department of Mathematics and Statistics,
STEM Building, 150 Louis-Pasteur Private, Ottawa, ON, K1N 6N5, Canada. E-mail address: bsaid053@uottawa.ca} \footnote{This article is based on a portion of the author's doctoral thesis.}}

\date{July 25, 2019}
\maketitle

\begin{abstract}
\noindent In this article, we introduce the space $\bD([0,1];\bD)$ of functions defined on $[0,1]$ with values in the Skorohod space $\bD$, which are right-continuous and have left limits with respect to the $J_1$ topology. This space is equipped with the Skorohod-type distance introduced in \cite{whitt80}. Following the classical approach of \cite{billingsley68, billingsley99}, we give several criteria for tightness of probability measures on this space, by characterizing the relatively compact subsets of this space. In particular, one of these criteria has been used in the recent article \cite{balan-saidani18} for proving the existence of a $\bD$-valued $\alpha$-stable L\'evy motion. Finally, we give a criterion for weak convergence of random elements in $\bD([0,1];\bD)$, and a
criterion for the existence of a process with sample paths in $\bD([0,1];\bD)$ based on its finite-dimensional distributions.
\end{abstract}



\vspace{1mm}


\section{Introduction}

The area of limit theorems for stochastic processes has been growing steadily in the last 50 years, especially after the publication of Billingsley's seminal monograph \cite{billingsley68}.
This area has developed from the original investigations of Donsker \cite{donsker51} and Skorohod \cite{skorohod56,skorohod57} regarding the asymptotic behaviour of the partial sum process associated with independent identically distributed (i.i.d.) random variables. These results state that if the variables have finite variance, the partial sum process converges in distribution to the Brownian motion, whereas if the variables have regularly varying tail probabilities with index $\alpha \in (0,2)$, the partial sum process converges in distribution to an $\alpha$-stable L\'evy motion.

In the recent article \cite{balan-saidani18}, we proved an extension of this later result to random elements with values in the Skorohod space $\bD=\bD[0,1]$ of c\`adl\`ag functions on $[0,1]$ (i.e. right-continuous functions with left limits), the limit being an infinite-dimensional process called the {\em $\bD$-valued $\alpha$-stable L\'evy motion}.
This study was motivated by the fact that nowadays, data is no longer observed at fixed moments at time, but is recorded continuously over a fixed period of time (or a region in space), which can be modeled by the interval $[0,1]$. This approach was initiated in article \cite{dehaan-lin01} in which the authors considered the example of the high-tide water level observed continuously at any location along
the northern coast of the Netherlands. Another example is the evolution of a stock price which is monitored continuously between 9 a.m and 5 p.m. when the stock market operates; if the price is likely to exhibit a sudden drop or increase, then its behaviour over this 8-hour window can be viewed as the sample path of a random process with values in $\bD$.

In \cite{balan-saidani18}, we proved that the $\bD$-valued $\alpha$-stable L\'evy motion $\{Z(t)\}_{t \in [0,1]}$ arises as the limit in distribution of the partial sum process $\{S_n(t)=a_n^{-1}\sum_{i=1}^{[nt]}X_i\}_{t \in [0,1]}$ associated with a sequence $(X_i)_{i\geq 1}$ of i.i.d. random elements in $\bD$, which are ``regularly varying'' in the sense introduced in \cite{hult-lindskog05}. Moreover, the sample paths of this process belong to the space $\bD([0,1];\bD)$ of functions defined on $[0,1]$ with values in $\bD$, which are right-continuous and have left-limits with respect to Skorohod $J_1$-topology on $\bD$. Therefore, the law of $\{Z(t)\}_{t \in [0,1]}$ is a probability measure on $\bD([0,1];\bD)$. For any element $x=\{x(t)\}_{t \in [0,1]}$ in $\bD([0,1];\bD)$, $x(t)$ is a c\`adl\`ag function denoted by $x(t)=\{x(t,s)\}_{s \in [0,1]}$. We interpret $t$ as the time variable and $s$ as the space variable.

The goal of the present article is to provide some of the technical details which are missing from the companion article \cite{balan-saidani18}, related to the weak convergence and tightness of probability measures on the space $\bD([0,1];\bD)$, providing in this way some useful tools for developing new limit theorems for random elements in $\bD$. In order to do this, we need first to develop a compactness criterion for subsets of $\bD([0,1];\bD)$. We note that the space $\bD([0,1];\bD)$ is endowed with a Skorohod-type topology which was introduced in \cite{whitt80} for spaces of the form $\bD([0,1];S)$, where $S$ is a Polish space, i.e. a complete separable metric space. The main result of the present article is Theorem \ref{analogue-th-13-2} which gives a
criterion for tightness of probability measures on $\bD([0,1];\bD)$. This result is new in the literature and has been used in the recent article \cite{balan-saidani18} for proving the existence of the $\bD$-valued $\alpha$-stable L\'evy motion with $\alpha>1$ (see the proof of Theorem 3.14 of \cite{balan-saidani18}).
The problem of weak convergence and tightness for probability measures on the space $\bD([0,\infty);S)$ of c\`adl\`ag functions defined on $[0,\infty)$ with values in a Polish space $S$ was also studied in \cite{ethier-kurtz86} (Chapter 3, Sections 5-9), but the particular result that we obtained when $S=\bD$ is not discussed in this reference.

Although it does not have a direct relationship with the results that we present here, we should mention that
a version of the It\^o-Nisio theorem for the sum of i.i.d. random processes with sample paths in $\bD([0,1];E)$ was proved in \cite{basse-rosinski13}, when $E$ is a separable Banach space. The space $\bD$ equipped with the uniform norm is a Banach space, but is not separable, and therefore the result of \cite{basse-rosinski13} does not apply to $\bD([0,1];\bD)$.

This article is organized as follows. In Section 2, we introduce the space $\bD([0,1];\bD)$ and discuss some of its properties. In Section 3, we present some criteria for tightness and weak convergence of probability measures on $\bD([0,1];\bD)$. One of these criteria, namely Theorem 3.8 below, has been used in the proof of Theorem 3.14 of \cite{balan-saidani18}. In Section 4, we refine the criterion for weak convergence and we derive a result about existence of a process with sample paths in $\bD([0,1];\bD)$. These results generalize classical results from \cite{billingsley68, billingsley99} and may be useful in future investigations.

\section{Basic properties of $\bD([0,1];\bD)$}

In this section, we introduce the space $\bD([0,1];\bD)$ of c\`adl\`ag functions on $[0,1]$ with values in $\bD$ (equipped with the $J_1$-topology), and we examine its properties following very closely the discussion contained in Section 12 of \cite{billingsley99} for c\`adl\`ag functions with values in $\bR$.


We begin by recalling some basic properties of the Skorohod space $\bD$, the space of functions $x:[0,1] \to \bR$ which are right-continuous and have left limits. On this space, we consider the supremum norm:
$\|x\|=\sup_{s \in [0,1]}|x(s)|$.

The Skorohod distance $d_{J_1}$ on $\bD$ is defined as follows: for any $x,y \in \bD$,
$$d_{J_1}(x,y)=\inf_{\lambda \in \Lambda} \{\|\lambda-e\| \vee \|x-y \circ \lambda\| \},$$
where $\Lambda$ the set of strictly increasing continuous functions from $[0,1]$ onto $[0,1]$ and $e$ is the identity function on $[0,1]$. The space $\bD$ equipped with distance $d_{J_1}$ is separable, but it is not complete. There exists another distance $d_{J_1}^0$ on $\bD$, which is equivalent to $d_{J_1}$, under which $\bD$ is complete and separable. This distance is given by:
(see (12.16) of \cite{billingsley99})
\begin{equation}
\label{def-dJ1}
d_{J_1}^0(x,y)=\inf_{\lambda \in \Lambda} \{\|\lambda\|^0 \vee \|x-y \circ \lambda\| \},
\end{equation}
for any $x,y \in \bD$, where
$\|\lambda\|^0=\sup_{s<s'} \left|\log \frac{\lambda(s')-\lambda(s)}{s'-s} \right|$.
Note that: 
\begin{equation}
\label{lambda0-equiv}
\sup_{s \in [0,1]}|\lambda(s)-s| \leq e^{\|\lambda\|^0}-1 \quad \mbox{for all} \ \lambda \in \Lambda,
\end{equation}
and therefore
\begin{equation}
\label{dJ1-smaller-dJ10}
d(x,y)\leq e^{d_{J_1}^0(x,y)}-1 \quad \mbox{for all} \quad x,y \in \bD.
\end{equation}
Taking $\lambda=e$ in \eqref{def-dJ1}, we obtain:
\begin{equation}
\label{dJ1-ineq}
d_{J_1}^0(x,y)\leq \|x-y\| \quad \mbox{for all} \quad x,y \in \bD.
\end{equation}
Note that
\begin{equation}
\label{d-equal-norm}
d_{J_1}(x,0)=d_{J_1}^0(x,0)=\|x\| \quad \mbox{for all} \quad x \in \bD.
\end{equation}
For functions $(x_n)_{n\geq 1}$ and $x$ in $\bD$, we write $x_n \stackrel{J_1}{\to}x$ if $d_{J_1}^0(x_n,x)\to 0$.

For any set $T \subset [0,1]$ and for any $x \in \bD(\bD)$, we let
$$w(x,T)=\sup_{s_1,s_2 \in T}|x(s_1)-x(s_2)|.$$

A set $\{t_i\}_{0\leq i\leq v}$ with $0=t_0<t_1<\ldots<t_v=1$ is called {\em $\delta$-sparse} if $\min_{1\leq i\leq v}(t_i-t_{i-1})>\delta$.
For any $\delta\in (0,1)$, we consider the following moduli of continuity of a function $x \in \bD$:
\begin{equation}
\label{def-w-prime}
w'(x,\delta)=\inf_{\{t_i\}} \max_{1\leq i\leq v} w(x,[t_{i-1},t_i)),
\end{equation}
where the infimum is taken over all $\delta$-sparse sets $\{t_i\}_{0\leq i \leq v}$, and
\begin{equation}
\label{def-w-second}
w''(x,\delta)=\sup_{s_1\leq s\leq s_2, s_2-s_1 \leq \delta} \big(|x(s)-x(s_1)|\wedge |x(s_2)-x(s)| \big).
\end{equation}

We denote by $\cD$ the Borel $\sigma$-field of $\bD$, which coincides with the $\sigma$-field generated by the projections $\pi_t:\bD \to \bR, t \in [0,1]$ given by $\pi_t(x)=x(t)$.

\medskip

We introduce now the set $\bD([0,1];\bD)$ of functions $x:[0,1]\to \bD$ such that:\\
{\em (i)} $x$ is right-continuous with respect to $J_1$, i.e. for any $t \in [0,1)$ and for any $(t_k)_{k\geq 1} \subset [0,1]$ with $t_k \to t$ and $t_k \geq t$ for all $k$, we have $x(t_k) \stackrel{J_1}{\to}x(t)$;\\
{\em (ii)} $x$ has left limits with respect to $J_1$, i.e. for any $t \in (0,1]$, there exists $x(t-) \in \bD$ such that for any $(t_k)_{k\geq 1} \subset [0,1]$ with $t_k \to t$ and $t_k <t$ for all $k$, we have $x(t_k) \stackrel{J_1}{\to}x(t-)$.

\vspace{2mm}
For any  $t \in [0,1]$, $x(t)$ is an element of $\bD$, which we denote by $\{x(t,s);s \in [0,1]\}$. In applications, $t$ may be interpreted as time variable, and $s$ as space variable (see \cite{balan-saidani18}).

The next result shows that a function in $\bD([0,1];\bD)$ is uniformly bounded in $t$ and $s$.

\begin{lemma}
For any $x \in \bD([0,1];\bD)$, the set $\{x(t);t\in [0,1]\}$ is relatively compact in $(\bD,J_1)$, and therefore $\sup_{t \in [0,1]}\|x(t)\|<\infty$.
\end{lemma}

\noindent {\bf Proof:} Let $A=\{x(t);t\in [0,1]\}$ and $\{x(t_n)\}_{n \geq 1}$ be an arbitrary sequence in $A$. There exists a monotone subsequence $(t_{n_k})_{k\geq 1}$: either $t_{n_k} \downarrow t$ or $t_{n_k} \uparrow t$. Then either $x(t_{n_k}) \stackrel{J_1}{\to} x(t)$ or $x(t_{n_k}) \stackrel{J_1}{\to}x(t-)$. This shows that any sequence in $A$ has a $J_1$-convergent subsequence. So, $A$ is relatively compact in $(\bD,J_1)$. The last part follows by the characterization of relative compactness in $(\bD,J_1)$ given by Theorem 12.3 of \cite{billingsley99}. $\Box$

\vspace{3mm}

We denote by $\|\cdot\|_{\bD}$ the {\em super-uniform norm} on $\bD([0,1];\bD)$ given by:
$$\|x\|_{\bD}=\sup_{t \in [0,1]}\|x(t)\|.$$

We let $d_{\bD}$ be the {\em Skorohod distance} on $\bD([0,1];\bD)$, given by relation (2.1) of \cite{whitt80}:
\begin{equation}
\label{def-d-D}
d_{\bD}(x,y)=\inf_{\lambda \in \Lambda} \{\|\lambda-e\| \vee \rho_{\bD}(x,y \circ \lambda) \},
\end{equation}
where $\rho_{\bD}$ is the {\em uniform distance} on $\bD([0,1];\bD)$:
\begin{equation}
\label{def-rho-D}
\rho_{\bD}(x,y)=\sup_{t \in [0,1]}d_{J_1}^{0}(x(t),y(t)).
\end{equation}

By relation \eqref{d-equal-norm}, it follows that for any $x \in \bD([0,1];\bD)$,
\begin{equation}
\label{dD-equal-norm}
d_{\bD}(x,0)=\rho_{\bD}(x,0)=\|x\|_{\bD}.
\end{equation}

Note that for any $x,y \in \bD([0,1];\bD)$, we have:
\begin{equation}
\label{dist-ineq1}
d_{\bD}(x,y)\leq \rho_{\bD}(x,y) \leq \|x-y\|_{\bD}.
\end{equation}

By definition, $d_{\bD}(x_n,x) \to 0$ if and only if there exists a sequence $(\lambda_n)_{n\geq 1} \subset \Lambda$ such that
\begin{equation}
\label{conv-dD}
\sup_{t \in [0,1]}|\lambda_n(t)-t| \to 0 \quad \mbox{and} \quad \sup_{t \in [0,1]}d_{J_1}^{0}(x_n(\lambda_n(t)),x(t))\to 0.
\end{equation}
Similarly to $\bD$, the uniform topology on $\bD([0,1];\bD)$
is stronger than the Skorohod topology on this space: if $\rho_{\bD}(x_n,x) \to 0$ then $d_{\bD}(x_n,x) \to 0$ (take $\lambda_n=e$ in \eqref{conv-dD}). The following result is also similar to the classical case.

\begin{lemma}
\label{properties-dD}
a) If $d_{\bD}(x_n,x) \to 0$, then $x_n(t) \stackrel{J_1}{\to} x(t)$ for any continuity point $t$ of $x$ (with respect to $J_1$).\\
b) If $d_{\bD}(x_n,x) \to 0$ and $x$ is continuous on $[0,1]$ with respect to $J_1$, then
$\rho_{\bD}(x_n,x) \to 0$.
\end{lemma}

\noindent {\bf Proof:} Let $(\lambda_n)_{n\geq 1}\subset \Lambda$ be such that \eqref{conv-dD} holds. a) Then
$$d_{J_1}^0(x_n(t),x(t))\leq d_{J_1}^0(x_n(t),x(\lambda_n(t)))+d_{J_1}^0(x(\lambda_n(t)),x(t))\to 0.$$
b) Since $x$ is continuous on the compact set $[0,1]$, it is also uniformly continuous. Hence
$$\rho_{\bD}(x_n,x) \leq \sup_{t \in [0,1]}d_{J_1}^0(x_n(t),x(\lambda_n(x(t))))+ \sup_{t \in [0,1]}d_{J_1}^0(x(\lambda_n(x(t))),x(t))\to 0.$$
$\Box$

The next result show that the super-uniform norm is continuous on $\bD([0,1];\bD)$.

\begin{lemma}
\label{dD-cont}
If $(x_n)_{n\geq 1}$ and $x$ are functions in $\bD([0,1];\bD)$ such that $d_{\bD}(x_n,x) \to 0$ as $n \to \infty$, then $\|x_n\|_{\bD} \to \|x\|_{\bD}$ as $n \to \infty$.
\end{lemma}

\noindent {\bf Proof:} Let $(\lambda_n)_{n \geq 1} \subset \Lambda$ be such that
\eqref{conv-dD} holds. By \eqref{dD-equal-norm}, we have:
$$|\|x_n \circ \lambda_n\|_{\bD} -\|x\|_{\bD}|=|\rho_{\bD}(x_n\circ \lambda_n,0)-\rho_{\bD}(x,0)|\leq \rho_{\bD}(x_n \circ \lambda_n,x) \to 0.$$
The conclusion follows since $\|x_n \circ \lambda_n\|_{\bD}=\|x_n\|_{\bD}$ (because $\lambda_n$ is a one-to-one map). $\Box$

\vspace{3mm}


For any set $T \subset [0,1]$ and for any $x \in \bD([0,1];\bD)$, we let
$$w_{\bD}(x,T)=\sup_{t_1,t_2 \in T}d_{J_1}^0(x(t_1),x(t_2)).$$

The following result is proved similarly to Lemma 1 (page 122) of \cite{billingsley99}.

\begin{lemma}
\label{w-D-e}
For any $x \in \bD([0,1];\bD)$ and $\e>0$, there exist $0=t_0<t_1<\ldots<t_v=1$ such that
$$w_{\bD}(x,[t_{i-1},t_i))<\e \quad \mbox{for all} \quad i=1,\ldots,v.$$
\end{lemma}

A consequence of this result is that for $x \in \bD([0,1];\bD)$ and $\e>0$, there can be at most finitely many points $t \in [0,1]$ such that
$d_{J_1}^0(x(t),x(t-))>\e$. Hence, any function $x \in \bD([0,1];\bD)$ has a countable set of discontinuities with respect to $J_1$, which we denote by ${\rm Disc}(x)$. The maximum jump of $x$ is defined by:
$$j(x)=\sup_{t \in [0,1]}d_{J_1}^0(x(t),x(t-))$$

For any $\delta\in (0,1)$ and $x \in \bD([0,1];\bD)$, we let
\begin{equation}
\label{def-wD-prime}
w_{\bD}'(x,\delta)=\inf_{\{t_i\}} \max_{1\leq i\leq v} w_{\bD}(x,[t_{i-1},t_i)),
\end{equation}
where the infimum is taken over all $\delta$-sparse sets $\{t_i\}_{0\leq i \leq v}$.

Clearly, the function $w'_{\bD}(x,\cdot)$ is non-decreasing. The following two results give some further properties of $w'_{\bD}(x,\delta)$.

\begin{lemma}
For any $x \in \bD([0,1];\bD)$,
\begin{equation}
\label{limit-w'-zero}
\lim_{\delta \to 0}w_{\bD}'(x,\delta)=0
\end{equation}
$$w'_{\bD}(x,\delta)\leq w_{\bD}(x,2\delta) \quad \mbox{for any} \quad \delta \in (0,1/2),$$
$$w_{\bD}(x,\delta)\leq 2w'_{\bD}(x,\delta)+j(x) \quad \mbox{for any} \quad \delta \in (0,1).$$
\end{lemma}

\noindent {\bf Proof:} To prove the first relation, let $\e>0$ be arbitrary and $\{t_i\}_{0\leq i\leq v}$ be the sequence given by Lemma \ref{w-D-e}. Pick $0<\delta_{\e}<\min_{0\leq i \leq v}(t_i-t_{i-1})$. For any $\delta \in (0,\delta_{\e})$, $\{t_i\}_{0\leq i\leq v}$ is $\delta$-sparse, and hence
$w_{\bD}'(x,\delta) \leq \max_{1\leq i\leq v} w_{\bD}(x,[t_{i-1},t_i))<\e$.
The last two relations are proved similarly to (12.7) and (12.9) of \cite{billingsley99}, using the triangle inequality in $(\bD,d_{J_1}^0)$. We omit the details. $\Box$

\begin{lemma}
\label{upper-semicont}
 $w'_{\bD}(\cdot,\delta)$ is upper-semicontinuous on $\bD([0,1];\bD)$ equipped with $d_{\bD}$.
\end{lemma}

\noindent {\bf Proof:} Let $x \in \bD([0,1];\bD)$ and $\e>0$ be arbitrary. We have to prove that there exists $\eta>0$ such that $w'_{\bD}(y,\delta)<w'_{\bD}(x,\delta)+\e$ for any $y \in
\bD([0,1];\bD)$ such that $d_{\bD}(x,y)<\eta$.
This follows by the same argument as in Lemma 4 (page 130) of \cite{billingsley99}, replacing $|y(t)-x(\lambda(t))|$ by $d_{J_1}^0(y(t),x(\lambda(t)))$ and using the triangle inequality in $(\bD,d_{J_1}^0)$. $\Box$

\vspace{3mm}

The space $\bD([0,1];\bD)$ equipped with $d_{\bD}$ is separable, but it is not complete. Similarly to the distance $d_{J_1}^0$ on $\bD$, we consider another distance $d_{\bD}^0$ on $\bD([0,1];\bD)$, given by:
\begin{equation}
\label{def-d-D-0}
d_{\bD}^0(x,y)=\inf_{\lambda \in \Lambda} \{\|\lambda\|^0 \vee \rho_{\bD}(x,y \circ \lambda) \}.
\end{equation}
Then
$d_{\bD}(x,y)\leq e^{d_{\bD}^0(x,y)}-1$ for all $x,y \in \bD([0,1];\bD)$.

Similarly to Theorems 12.1 and 12.2 of \cite{billingsley99}, and using the fact that $\bD$ is separable and complete under $d_{J_1}^0$, we obtain the following result. (See also Theorem 2.6 of \cite{whitt80}.)

\begin{theorem}
\label{bD-CSMS}
The metrics $d_{\bD}$ and $d_{\bD}^0$ are equivalent. The space $\bD([0,1];\bD)$ is separable under $d_{\bD}$ and $d_{\bD}^0$, and is complete under $d_{\bD}^0$.
\end{theorem}

The following result characterizes the relatively compact subsets of $\bD([0,1];\bD)$, being the analogue of Theorem 12.3 of \cite{billingsley99}.

\begin{theorem}
\label{analogue-th-12-3}
A set $A \subset \bD([0,1];\bD)$ is relatively compact with respect to $d_{\bD}$ if and only if it satisfies the following three conditions: \\
(i) $\sup_{x \in A}\|x\|_{\bD}<\infty$;\\
(ii) $\lim_{\delta \to 0}\sup_{x \in A} \sup_{t \in [0,1]}w'\big(x(t),\delta\big)=0$;\\
(iii) $\lim_{\delta \to 0} \sup_{x \in A} w'_{\bD}(x,\delta)=0$.
\end{theorem}

\noindent {\bf Proof:} Note that conditions {\em (i)} and {\em (ii)} are equivalent to saying that the set $U=\{x(t);x \in A,t \in [0,1]\}$ is relatively compact in $(\bD,J_1)$ (see Theorem 12.3 of \cite{billingsley99}).

Suppose that $A$ is relatively compact in $\bD([0,1];\bD)$. We first prove that $U$ is relatively compact in $(\bD,J_1)$. Let $\{x_n(t_n)\}_{n\geq 1}$ be an arbitrary sequence in $U$, with $x_n \in A$ and $t_n \in [0,1]$. Since $A$ is relatively compact, there exists a subsequence $N \subset \bN$ such that $d_{\bD}(x_n,x) \to 0$ as $n \to \infty,n \in N$. Let $(\lambda_n)_{n\geq 1} \subset \Lambda$ such that \eqref{conv-dD} hold as $n \to \infty,n \in N$. The sequence $(t_n)_{n \in N}$ has a monotone convergent sub-sequence $(t_n)_{n \in N'}$ with $N' \subset N$: either $t_n \uparrow t$ or $t_n \downarrow t$ as $n \to \infty,n \in N'$. Since $\lambda_n^{-1}$ is strictly increasing, either $\lambda_n^{-1}(t_n) \uparrow t$ or $\lambda_n^{-1}(t_n) \downarrow t$ as $n \to \infty,n\in N'$.
Therefore, either $x(\lambda_n^{-1}(t_n)) \stackrel{J_1}{\to} x(t-)$ or $x(\lambda_n^{-1}(t_n)) \stackrel{J_1}{\to} x(t)$ as $n \to \infty,n\in N'$. In the first case,
$$d_{J_1}^0\big(x_n(t_n),x(t-)\big) \leq d_{J_1}^0\big(x_n(t_n),x(\lambda_n^{-1}(t_n))\big)+d_{J_1}^0\big(x(\lambda_n^{-1}(t_n)),x(t-)\big)\to 0,$$
as $n \to \infty,n\in N'$. In the second case, $d_{J_1}^0\big(x_n(t_n),x(t)\big) \to 0$ as $n \to \infty,n\in N'$. This shows that the sequence $\{x_n(t_n)\}_{n\geq 1}$ has a $J_1$-convergence subsequence.

To prove {\em (iii)}, we apply Dini's theorem, as stated in Appendix M8 of \cite{billingsley99}.
Since $w_{\bD}'(\cdot,1/n)$ is upper semi-continuous for any $n$, and $w_{\bD}'(x,1/n) \downarrow 0$ for any $x \in \bD([0,1];\bD)$, this convergence is uniform on compact sets. Hence $\sup_{x \in A}w_{\bD}'(x,n^{-1})  \to 0$ as $n\to \infty$. Condition {\em (iii)} follows since
$w_{\bD}'(x,\cdot)$ is non-decreasing.

Next, suppose that the set $A$ satisfies conditions {\em (i)}-{\em (iii)}. Since $\bD([0,1];\bD)$ is complete with respect to $d_{\bD}^0$, the closure $\overline{A}$ of $A$ is also complete. To show that $\overline{A}$ is compact, it suffices to show that $\overline{A}$ is totally bounded with respect to $d_{\bD}^0$ (see Theorem of Appendix M5 of \cite{billingsley99}).
This follows as in the sufficiency part of the proof of Theorem 12.3 of \cite{billingsley99}, by choosing $H$ to be a finite $\e$-net of the set $U$ in $\bD$. $\Box$

\vspace{3mm}

To give a second characterization of the relatively compact subsets of $\bD([0,1];\bD)$, we consider the following modulus of continuity: for any $x \in \bD([0,1];\bD)$ and $\delta\in (0,1)$,
\begin{equation}
\label{def-wD-second}
w_{\bD}''(x,\delta)=\sup_{t_1\leq t\leq t_2,\, \, t_2-t_1 \leq \delta} \big(d_{J_1}^0(x(t),x(t_1))\wedge d_{J_1}^0(x(t_2),x(t)) \big).
\end{equation}

We have the following result.

\begin{lemma}[Lemma 2.2 of \cite{balan-saidani18}]
\label{dD-sum}
For any $x,y \in \bD([0,1];\bD)$, we have:
$$w_{\bD}''(x+y,\delta) \leq w_{\bD}''(x,\delta)+2\|y\|_{\bD}.$$
\end{lemma}

As in the classical case, it follows that $w_{\bD}''(x,\delta) \leq w_{\bD}'(x,\delta)$ (see the proof of (12.28) of \cite{billingsley99}). The following result is the analogue of Theorem 12.4 of \cite{billingsley99}.

\begin{theorem}
\label{analogue-th-12-4}
A set $A \subset \bD([0,1];\bD)$ is relatively compact with respect to $d_{\bD}$ if and only if it satisfies the following three conditions: \\
$(i)$ $\sup_{x \in A}\|x\|_{\bD}<\infty$;\\
$(ii')$
$$\left\{
\begin{array}{ll}
(a) & \lim_{\delta \to 0} \sup_{t \in [0,1]} w''(x(t),\delta) =0 \\
(b) & \lim_{\delta \to 0} \sup_{x \in A} \sup_{t \in [0,1]} |x(t,\delta),x(t,0)|=0 \\
(c) & \lim_{\delta \to 0} \sup_{x \in A} \sup_{t \in [0,1]} |x(t,1-),x(t,1-\delta)|=0;
\end{array}
\right.
$$
$(iii')$
$$\left\{
\begin{array}{ll}
(a) & \lim_{\delta \to 0} w_{\bD}''(x,\delta) =0 \\
(b) & \lim_{\delta \to 0} \sup_{x \in A} d_{J_1}^0(x(\delta),x(0))=0 \\
(c) & \lim_{\delta \to 0} \sup_{x \in A} d_{J_1}^0(x(1-),x(1-\delta))=0.
\end{array}
\right.
$$
\end{theorem}

\noindent {\bf Proof:} If $A$ is relatively compact, then conditions $(i)$-$(iii)$ of Theorem \ref{analogue-th-12-3} hold. Condition $(ii')$ follows by applying inequality (12.31) of \cite{billingsley99} to the function $x(t)\in \bD$, for any $t \in [0,1]$. Condition $(iii')$ follows by the following inequality (proved similarly to (12.31) of \cite{billingsley99}):
\begin{equation}
\label{analogue-12-31}
w_{\bD}''(x,\delta)\vee d_{J_1}^0\big(x(\delta),x(0) \big) \vee d_{J_1}^0\big(x(1-),x(1-\delta) \big)\leq w_{\bD}'(x,2\delta)
\end{equation}

Suppose that conditions $(i)$, $(ii')$ and $(iii')$ hold. The fact that $A$ is relatively compact will follow by Theorem \ref{analogue-th-12-3}, once we show that conditions $(ii)$ and $(iii)$ of  this theorem hold. Condition $(ii)$ follows from $(ii')$ by applying inequality (12.32) of \cite{billingsley99} to the function $x(t) \in \bD$ for any $t \in [0,1]$. Condition $(iii')$ follows by the following inequality
\begin{equation}
\label{analogue-12-32}
w_{\bD}'(x,\delta/2) \leq 12 \{w_{\bD}''(x,\delta)+ d_{J_1}^0\big(x(\delta),x(0) \big)+
d_{J_1}^0\big(x(1-),x(1-\delta) \big)\}.
\end{equation}
This is proved similarly to inequality (12.32) of \cite{billingsley99}, using the triangle inequality in $\bD$ and the fact that $x_n \stackrel{J_1}{\to}x$ implies that $d_{J_1}^0(x_n,y) \to d_{J_1}^0(x,y)$ for any $y \in \bD$.
$\Box$

\vspace{3mm}

We conclude this subsection with a discussion about measurability and finite-dimensional sets in $\bD([0,1];\bD)$.
Let $\cD_{\bD}$ be the Borel $\sigma$-field of $\bD([0,1];\bD)$ with respect to $d_{\bD}$. For any $t\in [0,1]$, we let $\pi_{t}^{\bD}:\bD([0,1];\bD) \to \bD$ be the projection given by $\pi_t^{\bD}(x)=x(t)$.
By Lemma 2.3 of \cite{whitt80}, $\pi_{t}^{\bD}$ is $\cD_{\bD}/\cD$-measurable for any $t \in [0,1]$. By Theorem 2.7 of \cite{whitt80}, $\cD_{\bD}$ coincides with the $\sigma$-field generated by the projections $\pi_{t}^{\bD}$ for $t \in [0,1]$. Similarly to the classical case, the function $\pi_t^{\bD}$ has the following continuity properties.

\begin{lemma}
a) $\pi_0^{\bD}$ and $\pi_1^{\bD}$ are continuous with respect to $d_{\bD}$. \\
b) For any $t \in (0,1)$, $\pi_t^{\bD}$ is continuous at $x$ with respect to $d_{\bD}$ if and only if $x$ is continuous at $t$ with respect to $J_1$.
\end{lemma}

\noindent {\bf Proof:} a) Assume that $d_{\bD}(x_n,x) \to 0$. Let $(\lambda_n)_{n\geq 1} \subset \Lambda$ be such that \eqref{conv-dD} holds. In particular, since $\lambda_n(0)=0$, we obtain: $d_{J_1}^0(x_n(0),x(0)) \to 0$. This shows that $\pi_0^{\bD}(x_n) \stackrel{J_1}{\to} \pi_0(x)$. Similarly, $\pi_1^{\bD}(x_n) \stackrel{J_1}{\to} \pi_1(x)$.

b) Suppose that $x$ is continuous at $t$ with respect to $J_1$. Assume that $d_{\bD}(x_n,x) \to 0$. Then $\pi_t^{\bD}(x_n) \stackrel{J_1}{\to} \pi_t^{\bD}(x)$, by Lemma \ref{properties-dD}.a).
Suppose next that $x$ is discontinuous at $t$ with respect to $J_1$, i.e. $d_{J_1}^0(x(t-),x(t))>0$. Let $\lambda_n \in \Lambda$ be such that $\lambda_n(t)=t-1/n$, and $\lambda$ is linear on $[0,t]$ and $[t,1]$. Define $x_n(s)=x(\lambda_n(s))$. Then $d_{\bD}(x_n,x) \to 0$, and $\pi_t^{\bD}(x_n)=x_n(t)=x(\lambda_n(t))=x(t-1/n) \stackrel{J_1}{\to} x(t-)$, and so $\pi_t^{\bD}(x_n)$ does not converge in $J_1$ to $x(t)$.
This shows that $\pi_t^{\bD}$ is discontinuous at $x$ with respect to $d_{\bD}$.
$\Box$

\vspace{3mm}

For an arbitrary set $T \subset [0,1]$, we let $\cD_{f,T}^{\bD}$ be the class of finite-dimensional sets of the form $(\pi_{t_1,\ldots,t_k}^{\bD})^{-1}(H)$ for some $0\leq t_1<\ldots<t_k\leq 1$, $t_i \in T$, $H \in \cD^k$ and $k\geq 1$. Note that the $\sigma$-field generated by $\cD_{f,T}^{\bD}$ coincides with $\sigma\{\pi_t^{\bD};t\in T\}$, the minimal $\sigma$-field with respect to which the maps $\pi_t^{\bD},t\in T$ are measurable.

\begin{theorem}
\label{separating-th}
If $T \subset [0,1]$ is such that $1 \in T$ and $T$ is dense in $[0,1]$, then:\\
a) $\cD_{\bD}$ is the $\sigma$-field generated by $\cD_{f,T}^{\bD}$;\\
b) $\cD_{f,T}^{\bD}$ is a separating class of $\cD_{\bD}$, i.e. if $P$ and $Q$ are two probability measures on $(\bD,\cD_{\bD})$ such that $P(A)=Q(A)$ for any $A \in \cD_{f,T}^{\bD}$, then $P=Q$.
\end{theorem}

\noindent {\bf Proof:} a) Since $\pi_{t}^{\bD}$ is $\cD_{\bD}$-measurable, $\sigma\{\pi_t^{\bD};t\in T\} \subset \cD_{\bD}$. To prove the other inclusion, it suffices to show that the identity $i: \bD([0,1];\bD) \to \bD([0,1];\bD)$ given by $i(x)=x$ is $\sigma\{\pi_t^{\bD};t \in [0,1]\}/\cD_{\bD}$-measurable.
For this, we use the same argument as in the proof of Theorem 12.5.(iii) of \cite{billingsley99}. For any $\sigma=\{t_i\}_{i=0,\ldots,k}$ such that $0=t_0<t_1<\ldots<t_k=1$, we define the map $A_{\sigma}: \bD([0,1];\bD) \to \bD([0,1];\bD)$ by $A_{\sigma}(x)=\sum_{i=1}^{k} x(t_{i-1}) 1_{[t_{i-1},t_i)}+x(1) 1_{\{1\}}(t)$.
Similarly to Lemma 3 (page 127) of \cite{billingsley99}, it can be proved that
\begin{equation}
\label{lemma3-p127}
\mbox{$\max_{1\leq i\leq k}(t_i-t_{i-1}) \leq \delta$ implies that $d_{\bD}(A_{\sigma}(x),x)\leq \delta \vee w_{\bD}'(x,\delta)$.}
\end{equation}

For any $\sigma$ as above, we consider also the map $V_{\sigma}:\bD^{k+1} \to \bD([0,1];\bD)$ given by
$V_{\sigma}(\alpha)=\sum_{i=1}^{k} \alpha_{i-1}1_{[t_{i-1},t_i)}(t)+\alpha_k 1_{\{1\}}(t)$, for $\alpha=(\alpha_0,\ldots,\alpha_k) \in \bD^{k+1}$.

The function $V_{\sigma}:\bD^{k+1} \to \bD([0,1];\bD)$ is $\rho_{\bD}$-continuous (hence $d_{\bD}$-continuous), where $\bD^{k+1}$ is endowed with the product topology: if $\alpha^n,\alpha \in \bD^k$ are such that $\alpha_i^n \stackrel{J_1}{\to} \alpha_i$ as $n \to \infty$, for $i=0,\ldots,k$, then
$$\rho_{\bD}(V_{\sigma}(\alpha^n),V_{\sigma}(\alpha))=\sup_{t \in [0,1]} d_{J_1}^0\big(V_m(\alpha^n)(t),V_m(\alpha)(t) \big)=\max_{0\leq i\leq k} d_{J_1}^0(\alpha_i^n,\alpha_i)\to 0.$$
It follows that $V_{\sigma}$ is $\cD^{k+1}/\bD_{\bD}$-measurable. If $t_i \in T$ for all $i$, then $A_{\sigma}$ is $\sigma\{\pi_t^{\bD};t \in T\}/\cD_{\bD}$-measurable, since $A_{\sigma}=V_{\sigma} \circ \pi_{t_0,\ldots,t_k}^{\bD}$ and $\pi_{t_0,\ldots,t_k}^{\bD}$ is $\sigma\{\pi_t^{\bD};t \in T\}/\cD_{\bD}^{k+1}$-measurable.

For any $m \geq 1$, choose $\sigma_m=\{t_i^m\}_{i=0,\ldots,k_m}$ such that $t_i^m \in T$ and $\max_{i}(t_i^m-t_{i-1}^m)<1/m$. By \eqref{limit-w'-zero} and \eqref{lemma3-p127}, it follows that $d_{\bD}\big(A_{\sigma_m}(x),x \big) \to 0$ as $m \to \infty$. This proves that the identity map $i$ is the pointwise limit (with respect to $d_{\bD}$) of the sequence $(A_{\sigma_m})_{m\geq 1}$ of $\sigma\{\pi_t^{\bD};t \in T\}/\cD_{\bD}$-measurable maps. Since $\bD_{\bD}$ is the Borel $\sigma$-field corresponding to $d_{\bD}$, it the map $i$ is also $\sigma\{\pi_t^{\bD};t \in T\}/\cD_{\bD}$-measurable.

b) This follows by Theorem 3.3 of \cite{billingsley95}, since $\cD_{f,T}^{\bD}$ is a $\pi$-system generating $\cD_{\bD}$. $\Box$

\vspace{3mm}
The characterization of tightness of probability measures on $\bD([0,1];\bD)$ given in Section \ref{section-tight}
relies on certain events involving the functions $w_{\bD}'(\cdot,\delta)$ and $w_{\bD}''(\cdot,\delta)$.
Measurability of these functions is essential for this purpose. Before establishing this, we need the following simple result (which is valid in any metric space).

\begin{lemma}
\label{cont-Phi}
The map $\Phi: \bD \times \bD \to [0,\infty)$ given by $\Phi(x,y)=d_{J_1}^0(x,y)$ is continuous with respect to the product of $J_1$-topologies on $\bD \times \bD$.
\end{lemma}

\noindent {\bf Proof:} If $x_n \stackrel{J_1}{\to} x$ and $y_n \stackrel{J_1}{\to}y$, then $d_{J_1}^0(x_n,y_n) \to d_{J_1}^0(x,y)$ since
\begin{align*}
|d_{J_1}^0(x_n,y_n) - d_{J_1}^0(x,y)| & \leq |d_{J_1}^0(x_n,y_n) - d_{J_1}^0(x,y_n)|+
|d_{J_1}^0(x,y_n) - d_{J_1}^0(x,y)| \\
&\leq d_{J_1}^0(x_n,x)+d_{J_1}^0(y_n,y).
\end{align*}
$\Box$

\begin{lemma}
The functions $w_{\bD}'(\cdot,\delta)$ and $w_{\bD}''(\cdot,\delta)$ are $\cD_{\bD}$-measurable.
\end{lemma}

\noindent {\bf Proof:} The measurability of $w_{\bD}'(\cdot,\delta)$ follows by Lemma \ref{upper-semicont}. For
$w_{\bD}''(\cdot,\delta)$, note that in the definition \eqref{def-wD-second} of $w_{\bD}''(x,\delta)$, we may take $t_1,t,t_2$ to be rational numbers. By Lemma \ref{cont-Phi}, $\Phi$ is $\cD \times \cD$-measurable, and so the map $\Phi \circ \pi_{t,t_1}^{\bD}$ given by $x \mapsto d_{J_1}^0(x(t),x(t_1))$ is $\cD_{\bD}$-measurable, for any $t_1,t \in [0,1]$. Therefore, the map
$x \mapsto d_{J_1}^0(x(t),x(t_1)) \wedge d_{J_1}^0(x(t_2),x(t))$ is $\cD_{\bD}$-measurable, for any rational numbers $t_1,t,t_2 \in [0,1]$ with $t_1\leq t \leq t_2$. The conclusion follows since the supremum of a countable collection of measurable functions is measurable. $\Box$

\vspace{3mm}

Finally, we recall the definition of a random element in $\bD([0,1];\bD)$.

\begin{definition}
{\rm Let $(\Omega,\cF,P)$ be a probability space. A map $X: \Omega \to \bD([0,1];\bD)$ is called a {\em random element} in $\bD([0,1];\bD)$ if $X$ is $\cF/\cD_{\bD}$-measurable, i.e. $X(t)$ is $\cF/\cD$-measurable for any $t \in [0,1]$.
}
\end{definition}

\section{Weak convergence and Tightness}
\label{section-tight}

In this section, we study the weak convergence and tightness of probability measures on the space $\big(\bD([0,1];\bD),\cD_{\bD}\big)$, following the discussion contained in Section 13 of \cite{billingsley99} for probability measures on $(\bD,\cD)$. We provide some of the details which are missing from \cite{billingsley99}), since they are more delicate and require special attention in our situation.

Recall that if $(P_n)_{n \geq 1}$ and $P$ are probability measures on $\big(\bD([0,1];\bD),\cD_{\bD}\big)$, we say that $(P_n)_{n \geq 1}$ {\em converges weakly} to $P$ if $\int f dP_n \to \int f dP$ for any $d_{\bD}$-continuous bounded function $f:\bD([0,1];\bD) \to \bR$. In this case, we write $P_n \stackrel{w}{\to} P$. Since $\bD([0,1];\bD)$ is separable, there is a distance on the set of probability measures on $\big(\bD([0,1];\bD),\cD_{\bD}\big)$ (called the {\em Prohorov distance}), which gives rise to the topology of weak convergence (see page 72 of \cite{billingsley99}).

If $(X_n)_{n\geq 1}$ and $X$ are random elements in $\bD([0,1];\bD)$ (possibly defined on different probability spaces) with respective laws denoted by $(P_n)_{n\geq 1}$ and $P$, we say that $(X_n)_{n\geq 1}$ {\em converges in distribution} to $X$ if $P_n \stackrel{w}{\to}P$. In this case, we write $X_n \stackrel{d}{\to}X$.

For any probability measure $P$ on $\big(\bD([0,1];\bD),\cD_{\bD}\big)$, we let $T_P$ be the set of $t \in [0,1]$ for which the projection $\pi_t^{\bD}$ is $d_{\bD}$-continuous a.s. with respect to $P$. Note that $0,1 \in T_P$. If $t \in (0,1)$, then $t \in T_P$ if and only if $P(J_t)=0$, where $J_t=\{x \in \bD([0,1];\bD); t \in {\rm Disc}(x)\}$.

Using the same argument as in the classical case (page 238 of \cite{billingsley99}), it can be shown that $P(J_t)>0$ is possible for at most countably many $t \in (0,1)$. Hence, the complement of $T_P$ in $[0,1]$ is countable. The following result follows by the continuous mapping theorem.

\begin{lemma}
\label{projection-conv}
Let $(P_n)_{n \geq 1}$ and $P$ be probability measures on $\big(\bD([0,1];\bD),\cD_{\bD}\big)$ such that $P_n \stackrel{w}{\to} P$. Then $P_n \circ (\pi_{t_1,\ldots,t_k}^{\bD})^{-1} \stackrel{w}{\to} P \circ (\pi_{t_1,\ldots,t_k}^{\bD})^{-1}$ for any $t_1,\ldots,t_k \in T_P$.
\end{lemma}

We recall the following definitions.

\begin{definition}
{\rm A family $\Pi$ of probability measures on $\big(\bD([0,1];\bD),\cD_{\bD}\big)$ is {\em tight} if for every $\eta>0$, there exists a $d_{\bD}$-compact set $K$ in $\bD([0,1];\bD)$ such that $P(K)\geq 1-\eta$ for all $P \in \Pi$.

}
\end{definition}

\begin{definition}
{\rm A family $\Pi$ of probability measures on $\big(\bD([0,1];\bD),\cD_{\bD}\big)$ is {\em relatively compact} if for every sequence $(P_n)_{n\geq 1}$ in $\Pi$, there exists a subsequence $(P_{n_k})_{k\geq 1}$ which converges weakly to a probability measure $Q$ (which is not necessarily an element of $\Pi$).
}
\end{definition}

The following result follows by Prohorov's theorem, since
$\bD([0,1];\bD)$ is separable and complete (see Theorems 5.1 and 5.2 of \cite{billingsley99}).

\begin{theorem}
\label{prohorov-th}
A family $\Pi$ of probability measures on $\big(\bD([0,1];\bD),\cD_{\bD}\big)$ is tight if and only if it is relatively compact.
\end{theorem}

The next result is an important tool for proving weak convergence in $\bD([0,1];\bD)$. Its proof is the same as in the classical case (see Theorem 13.1 of \cite{billingsley99}). We include it for the sake of completeness.

\begin{theorem}
\label{analogue-th-13-1}
Let $(P_n)_{n \geq 1}$ and $P$ be probability measures on $\big(\bD([0,1];\bD),\cD_{\bD}\big)$ such that
\begin{equation}
\label{fin-dim-conv}
P_n \circ (\pi_{t_1,\ldots,t_k}^{\bD})^{-1} \stackrel{w}{\to} P \circ (\pi_{t_1,\ldots,t_k}^{\bD})^{-1}  \ \mbox{in} \ \bD^k,
\ \mbox{for any} \ t_1,\ldots,t_k \in T_P
\end{equation}
and $(P_n)_{n\geq 1}$ is tight.
Then $P_n \stackrel{w}{\to} P$.
\end{theorem}

\noindent {\bf Proof:} It is enough to prove that for any subsequence $(n_k)_{k \geq 1}$, there exists a further sub-subsequence $(k_l)_{l\geq 1}$ such that $P_{n_{k_l}} \stackrel{w}{\to}P$ as $l \to \infty$ (see e.g. Appendix 5.1.2 of \cite{kallenberg83}).

Let $(n_k)_{k\geq 1}$ be an arbitrary subsequence. By Theorem \ref{prohorov-th}, $(P_n)_{n\geq 1}$ is relatively compact. Hence, there exists a sub-subsequence $(k_l)_{l \geq 1}$ such that $P_{n_{k_l}} \stackrel{w}{\to}Q$ as $l\to \infty$, for some probability measure $Q$ on $\big(\bD([0,1];\bD),\cD_{\bD}\big)$. By hypothesis,
$P_{n_{k_l}} \circ (\pi_{t_1,\ldots,t_k}^{\bD})^{-1} \stackrel{w}{\to} P \circ (\pi_{t_1,\ldots,t_k}^{\bD})^{-1}$ as $l\to \infty$, for any $t_1,\ldots,t_k \in T_P$.
By Lemma \ref{projection-conv}, $P_{n_{k_l}} \circ (\pi_{t_1,\ldots,t_k}^{\bD})^{-1} \stackrel{w}{\to} Q \circ (\pi_{t_1,\ldots,t_k}^{\bD})^{-1}$ as $l\to \infty$, for any $t_1,\ldots,t_k \in T_Q$. Uniqueness of the limit implies that:
$$P \circ (\pi_{t_1,\ldots,t_k}^{\bD})^{-1}=Q \circ (\pi_{t_1,\ldots,t_k}^{\bD})^{-1} \quad \mbox{for all} \
t_1,\ldots,t_k \in T_P \cap T_Q.$$
The set $T=T_{P}\cap T_{Q}$ contains $0$ and $1$, and is dense in $[0,1]$ (since its complement in $[0,1]$ is countable). By Theorem \ref{separating-th}, $\cD_{f,T}$ is a separating class of $\cD_{\bD}$, and hence $P=Q$. $\Box$

\vspace{3mm}

We continue now with a discussion about tightness. The next result gives a criterion for tightness, being the analogue of Theorem 13.2 of \cite{billingsley99} for the space $\bD([0,1];\bD)$. This result has been used in the recent article \cite{balan-saidani18} for the construction of the $\bD$-valued $\alpha$-stable L\'evy motion with $\alpha>1$ (see the proof of Theorem 3.14 of \cite{balan-saidani18}).
Conditions {\em (i)} and {\em (iii)} of this result are similar to (13.4) and (13.5) of \cite{billingsley99}, but {\em (ii)} is a new condition, due to the space variable $s$ of an element in $\bD([0,1];\bD)$.
Recall that $w'\big(x(t),\delta\big)$ is given by \eqref{def-w-prime}, whereas $w_{\bD}'(x,\delta)$ is given by \eqref{def-wD-prime}, for any $x \in \bD([0,1];\bD)$ and $t \in [0,1]$.

\begin{theorem}
\label{analogue-th-13-2}
A sequence $(P_n)_{n\geq 1}$ of probability measures on $\big(\bD([0,1];\bD),\cD_{\bD}\big)$ is tight if and only if it satisfies the following three conditions:\\
(i) We have:
\begin{equation}
\label{analogue-13-4}
\lim_{a \to \infty} \limsup_{n\to \infty} P_n\big(\{x;\, \|x\|_{\bD} \geq a\}\big)=0.
\end{equation}
(ii) For any $\e>0$,
\begin{equation}
\label{new-cond-tight}
\lim_{\delta \to 0}\limsup_{n\to \infty}P_n\big(\{x;\,w'\big(x(t),\delta\big) \geq \e \ \mbox{for some} \ t \in [0,1]\}\big)=0.
\end{equation}
(iii) For any $\e>0$,
\begin{equation}
\label{analogue-13-5}
\lim_{\delta \to 0}\limsup_{n\to \infty}P_n\big(\{x;\,w_{\bD}'(x,\delta) \geq \e \}\big)=0.
\end{equation}
\end{theorem}

\noindent{\bf Proof:} We use a similar argument as in the proof of Theorem 13.2 of \cite{billingsley99} (see also the proof of Theorem 7.3 of \cite{billingsley99}). Suppose that $(P_n)_{n\geq 1}$ is tight.
Let $\eta>0$ and $\e>0$ be arbitrary. We have to prove that there exist $a>0$, $\delta \in (0,1)$ and an integer $n_0 \geq 1$ such that for all $n \geq n_0$,
\begin{equation}
\label{tight-abc}
\left\{
\begin{array}{ll}
(a) & P_n\big(\{x;\, \|x\|_{\bD} \geq a\}\big) \leq \eta  \\
(b) & P_n\big(\{x;\,w'\big(x(t),\delta\big) \geq \e \ \mbox{for some} \ t \in [0,1]\}\big) \leq \eta \\
(c) & P_n\big(\{x;\,w_{\bD}'(x,\delta) \geq \e \big) \leq \eta.
\end{array}
\right.
\end{equation}
We will show that {\em (a)}-{\em (c)} hold with $n_0=1$. By Theorem \ref{prohorov-th}, $(P_n)_{n\geq 1}$ is relatively compact. Hence, there exists a compact set $K$ in $\bD([0,1];\bD)$ such that $P_n(K)\geq 1-\eta$ for all $n \geq 1$. The set $K$ is characterized using Theorem \ref{analogue-th-12-3}. More precisely, we know that:
\begin{equation}
\label{compact-cond}
\left\{
\begin{array}{ll}
(a') & \sup_{x \in K}\|x\|_{\bD}<\infty  \\
(b') & \lim_{\delta \to 0}\sup_{x \in K}\sup_{t \in [0,1]}w'\big(x(t),\delta\big)=0\\
(c') & \lim_{\delta \to 0}\sup_{x\in K} w_{\bD}'(x,\delta)=0
\end{array}
\right.
\end{equation}
Due to $(a')$, we can choose $a > \sup_{x \in K}\|x\|_{\bD}$ arbitrary. Then $K \subset \{x; \|x\|_{\bD} < a\}$ and so, $$P_n \big(\{x;\, \|x\|_{\bD} \geq a\}\big) \leq P_n(K^c)\leq \eta \quad \mbox{for all} \ n \geq 1.$$
By $(b')$, there exists $\delta \in (0,1)$ such that $w'(x(t),\delta)<\e$ for all $x \in K,t \in [0,1]$. Hence, $K \subset \{x;\,w'(x(t),\delta)<\e  \ \mbox{for all} \ t \in [0,1]\}$, and so
$$P_n\big( \{x;\,w'\big(x(t),\delta\big)<\e  \ \mbox{for some} \ t \in [0,1]\}\big) \leq P_n(K^c) \leq \eta \quad \mbox{for all} \ n \geq 1.$$
By $(c')$, there exists $\delta \in (0,1)$ such that $w_{\bD}'(x,\delta)<\e$ for all $x \in K$. Hence, $K \subset \{x;\,w_{\bD}'(x,\delta)<\e\}$, and so
$$P_n\big( \{x;\,w'\big(x,\delta\big)<\e  \}\big) \leq P_n(K^c) \leq \eta \quad \mbox{for all} \ n \geq 1.$$

Suppose next that conditions {\em (i)}-{\em (iii)} hold. Let $\eta>0$ and $\e>0$ be arbitrary. Then there exist $a'>0$, $\delta' \in (0,1)$ and an integer $n_0\geq 1$ such that \eqref{tight-abc} holds for all $n\geq n_0$
(with $a'$ and $\delta'$ replacing $a$ and $\delta$).
We first prove that \eqref{tight-abc} actually holds for all $n \geq 1$, for some values $a$ and $\delta$ which will be given below. Fix $i \in \{1,\ldots,n_0-1\}$. Since $\bD([0,1];\bD)$ is separable and complete, the single probability measure $P_i$ is tight, and therefore it satisfies conditions {\em (i)}-{\em (iii)}. Hence, there exists $a_i>0$ and $\delta_i \in (0,1)$ such that
$$\left\{
\begin{array}{ll}
& P_i\big(\{x;\, \|x\|_{\bD} \geq a_i\}\big) \leq \eta  \\
& P_i\big(\{x;\,w'\big(x(t),\delta_i\big) \geq \e \ \mbox{for some} \ t \in [0,1]\}\big) \leq \eta \\
& P_i\big(\{x;\,w_{\bD}'(x,\delta_i) \geq \e \big) \leq \eta.
\end{array}
\right.$$
Then \eqref{tight-abc} holds for all $n\geq 1$, with $a=\max\{a',\max_{i \leq n_0-1}a_i\}$ and $\delta=\min\{\delta',\min_{i\leq n_0-1}\delta_i\}$.

Let $B=\{x;\|x\|_{\bD} < a\}$. Then $P_n(B) \geq 1-\eta$ for all $n\geq 1$. By parts {\em (b)} and {\em (c)} of \eqref{tight-abc} with $\e=1/k$ and $\eta$ replaced by $\eta/2^k$, there exists $\delta_k \in (0,1)$ such that for all $n\geq 1$,
$$P_n(B_k) \geq 1-\frac{\eta}{2^k} \quad \mbox{and} \quad P_n(C_k) \geq 1-\frac{\eta}{2^k},$$ where $B_k=\{x;  \sup_{t \in [0,1]}w'(x(t),\delta_k)<1/k \}$ and $C_k=\{x;w_{\bD}'(x,\delta_k)<1/k\}$. Let $A=B \cap \big(\cap_{k\geq 1}B_k\big) \cap \big(\cap_{k\geq 1} C_k\big)$ and $K=\overline{A}$. For any $n\geq 1$, $P_n(K) \geq P_n(A) \geq 1-3\eta$, since
$$P_n(A^c) \leq   P_n(B^c) +\sum_{k\geq 1}P_n(B_k^c)+\sum_{k\geq 1}P_n(C_k^c) \leq \eta+\sum_{k\geq 1}\frac{\eta}{2^k}+\sum_{k\geq 1}\frac{\eta}{2^k}=3\eta.$$
We show that $K$ is compact in $\bD([0,1];\bD)$. By Theorem \ref{analogue-th-12-3}, this is equivalent to showing that $K$ satisfies \eqref{compact-cond}. Since $\|x\|_{\bD}< a$ for any $x \in B$ and $A \subset B$, we have $\sup_{x \in A}\|x\|_{\bD} <a$. This shows that $(a')$ holds. Note that for any $k\geq 1$, $\sup_{x\in A}\sup_{t \in [0,1]}w'(x(t),\delta_k)<1/k$ (since $A \subset B_k$), and so $(b')$ holds. Finally, for any $k\geq 1$, $\sup_{x \in A}w_{\bD}'(x,\delta_k)<1/k$ (since $A \subset C_k$), and hence $(c')$ holds. This proves that $(P_n)_{n\geq 1}$ is tight. $\Box$

\vspace{3mm}

The following result gives a replacement for condition {\em (i)} in Theorem \ref{analogue-th-13-2}. This condition is the analogue of (13.6) of \cite{billingsley99}.

\begin{corollary}
\label{analogue-corol-p140}
Condition (i) of Theorem \ref{analogue-th-13-2} can be replaced by the following condition:\\
(i') for each $t$ in a dense subset $T$ of $[0,1]$ which contains $1$, we have:
\begin{equation}
\label{analogue-13-6}
\lim_{a \to \infty} \limsup_{n\to \infty} P_n\big(\{x;\|x(t)\|\geq a\}\big)=0.
\end{equation}
\end{corollary}

\noindent {\bf Proof:} Suppose that condition $(i)$ of Theorem \ref{analogue-th-13-2} holds.
Then $(i')$ clearly holds, since $\{x;\|x(t)\|\geq a\} \subset \{x;\|x\|_{\bD}\geq a\}$ for any $t \in T$.

Suppose next that conditions $(i')$ and $(iii)$ hold. We prove that $(i)$ holds, using a similar argument as in the Corollary on page 140 of \cite{billingsley99}. Let $\eta>0$ be arbitrary. By condition $(iii)$, there exist $\delta \in (0,1)$ and an integer $n_1 \geq 1$ such that
\begin{equation}
\label{Pn-wD}
P_n\big(\{x; w_{\bD}'(x,\delta)\geq 1\}\big) \leq \eta \quad \mbox{for all} \ n \geq n_1.
\end{equation}
Let $\{t_i\}_{i=1,\ldots,v}$ be a $\delta$-sparse set with $0=t_0<t_1<\ldots<t_v=1$ such that $w_{\bD}(x,[t_{i-1},t_i)) \leq w_{\bD}'(x,\delta)+1$ for all $i=1,\ldots,v$. Choose points $0=s_0<s_1<\ldots <s_k=1$ such that $s_j \in T$ and $s_j-s_{j-1}<\delta$ for all $k=1,\ldots,k$. Let $m(x)=\max_{1 \leq j\leq k}\|x(s_j)\|$. By \eqref{analogue-13-6},
$\lim_{a \to \infty} \limsup_{n\to \infty} P_n\big(\{x;m(x)\geq a\}\big)=0$.
So, there exist $a>0$ and $n_2 \geq 1$ such that
\begin{equation}
\label{Pn-mx}
P_n\big(\{x; m(x)\geq a\}\big)\leq \eta \quad \mbox{for all $n\geq n_2$}.
\end{equation}
We claim that for any $x \in \bD([0,1];\bD)$,
\begin{equation}
\label{norm-x-m-ineq}
\|x\|_{\bD} \leq w_{\bD}'(x,\delta)+1+m(x).
\end{equation}
To see this, note that since $\{t_i\}_{i}$ is $\delta$-sparse, each interval $[t_{i-1},t_i)$ contains at least one point $s_{j}$, that we call $s_{j_i}$. For any $i=1,\ldots,v$ and for any $t \in [t_{i-1},t_i)$,
$$\|x(t)\|=d_{J_1}^0\big(x(t),0\big) \leq d_{J_1}^0\big(x(t),x(s_{j_i})\big)+ d_{J_1}^0\big(x(s_{j_i}),0\big)=
 d_{J_1}^0\big(x(t),x(s_{j_i})\big)+ \|x(s_{j_i})\|.$$
Hence,
$$\sup_{t \in [t_{i-1},t_i)}\|x(t)\| \leq w_{\bD}(x,[t_{i-1},t_i))+\|x(s_{j_i})\| \leq w_{\bD}'(x,\delta)+1+m(x).$$
Relation \eqref{norm-x-m-ineq} follows since $\|x\|_{\bD} =\max \{ \max_{1 \leq i\leq v} \sup_{t \in [t_{i-1},t_i)} \|x(t)\|,\|x(1)\| \}$.

Let $n_0=\max(n_1,n_2)$. From \eqref{Pn-wD}, \eqref{Pn-mx} and \eqref{norm-x-m-ineq}, we infer that
$$P_n\big(\{x;\|x\|_{\bD}\geq a+2\}\big) \leq P_n\big(\{x;w_{\bD}'(x,\delta)+m(x)\geq a+1\}\big)\leq 2
\eta \quad \mbox{for all $n\geq n_0$}.$$
This concludes the proof of $(i)$. $\Box$

\vspace{3mm}

The following result is the analogue of relation (13.8) of \cite{billingsley99} (or Theorem 15.3 of \cite{billingsley68}), and it plays a crucial role in article \cite{balan-saidani18} (see Theorem 2.4 of \cite{balan-saidani18}).

\begin{theorem}
\label{analogue-13-8}
A sequence $(P_n)_{n\geq 1}$ of probability measures on $\big(\bD([0,1];\bD),\cD_{\bD}\big)$ is tight if and only if it satisfies condition (i) of Theorem \ref{analogue-th-13-2} and the following two conditions:\\
$(ii')$ For any $\e>0$,
$$
\left\{
\begin{array}{ll}
(a) & \lim_{\delta \to 0}\limsup_{n \to \infty}P_n(\{x;\,w''(x(t),\delta) \geq \e \ \mbox{for some} \ t \in [0,1]\})=0; \\
(b) & \lim_{\delta \to 0}\limsup_{n \to \infty}P_n(\{x;\,|x(t,\delta)-x(t,0)| \geq  \e \ \mbox{for some} \ t \in [0,1]\})=0; \\
(c) & \lim_{\delta \to 0}\limsup_{n \to \infty}P_n(\{x;\,|x(t,1-)-x(t,1-\delta)| \geq  \e \ \mbox{for some} \ t \in [0,1]\})=0.
\end{array}
\right.$$
$(iii')$  For any $\e>0$,
$$
\left\{
\begin{array}{ll}
(a) & \lim_{\delta \to 0}\limsup_{n \to \infty}P_n(\{x;\,w_{\bD}''(x,\delta) \geq \e \})=0; \\
(b) & \lim_{\delta \to 0}\limsup_{n \to \infty}P_n(\{x;\,d_{J_1}^{0}\big(x(\delta),x(0)\big) \geq \e \})=0; \\
(c) & \lim_{\delta \to 0}\limsup_{n \to \infty}P_n(\{x;\,d_{J_1}^{0}\big(x(1-),x(1-\delta)\big) \geq \e \})=0.
\end{array}
\right.$$
\end{theorem}

\begin{proof}
This follows directly from Theorem \ref{analogue-th-13-2}. To see this, note that $(ii')$ is equivalent to $(ii)$ of Theorem \ref{analogue-th-13-2}, due to inequalities (12.31) and (12.32) of \cite{billingsley99}), whereas $(iii')$ is equivalent to $(iii)$ of Theorem \ref{analogue-th-13-2}, due to inequalities \eqref{analogue-12-31} and \eqref{analogue-12-32}.
\end{proof}

The  following result is the analogue of Theorem 13.3 of \cite{billingsley99}.

\begin{theorem}
\label{analogue-th-13-3}
Let $(P_n)_{n\geq 1}$ and $P$ be probability measures on $\bD([0,1];\bD)$ such that \eqref{fin-dim-conv} holds,
$(P_n)_{n\geq 1}$ satisfies parts $(ii')$ and $(iii'.a)$ of Theorem \ref{analogue-13-8}, and $P$ satisfies
\begin{equation}
\label{analogue-13-9}
\lim_{\delta \to 0} P\big(\{x; \, d_{J_1}^0\big(x(1),x(1-\delta)\big)\geq \e \} \big)=0 \quad \mbox{for all} \ \e>0.
\end{equation}
Then $P_n \stackrel{w}{\to} P$.
\end{theorem}

\noindent {\bf Proof:} By Theorem 13.1, it is enough to prove that $(P_n)_{n\geq 1}$ is tight. For this, we use Theorem \ref{analogue-13-8}. We first check condition $(i')$ given by Corollary \ref{analogue-corol-p140}, with $T=T_P$. Let $t \in T_P$ be arbitrary. The sequence $\{P_n \circ (\pi_t^{\bD})^{-1}\}_{n\geq 1}$ is relatively compact in $\bD$ being weakly convergent. By Prohorov theorem, this sequence is tight. Hence, for any $\eta>0$, there exists a compact set $K$ in $\bD$ such that $[P_n \circ (\pi_t^{\bD})^{-1}](K^c) \leq \eta$ for all $n \geq 1$. By Theorem 12.3 of \cite{billingsley99}, $M:=\sup_{y \in K}\|y\|<\infty$. For any $a>M$,  $\{y\in \bD; \|y\|\geq a\} \subset K^c$ and
$$P_n\big(\{x; \|x(t)\|\geq a \}\big)
\leq [P_n \circ (\pi_t^{\bD})^{-1}](K^c) \leq \eta \quad \mbox{for all} \ n\geq 1.$$

Next, we check that part (b) of $(iii')$ holds. Let $\e>0$ and $\eta>0$ be arbitrary. By the right continuity of elements in $\bD([0,1];\bD)$,
$P\big(\{x; d_{J_1}^0\big(x(\delta),x(0) \big)\geq \e \}\big) \to 0$ as $\delta \to 0$. Choose $\delta \in T_P$ small such that $P\big(\{x; d_{J_1}^0\big(x(\delta),x(0) \big) \}\big)<\eta$. By \eqref{fin-dim-conv}, $P_n \circ (\pi_{0,\delta}^{\bD})^{-1} \stackrel{w}{\to} P \circ (\pi_{0,\delta}^{\bD})^{-1}$ in $\bD^2$.
By Lemma \ref{cont-Phi}, the set $A=\{(y_1,y_2) \in \bD^2; d_{J_1}^0(y_1,y_2) \geq \e\}$ is closed in $\bD^2$ with respect to the product of $J_1$-topologies. By Portmanteau theorem, it follows that
$$\limsup_{n \to \infty }P_n\big(\{x; d_{J_1}^0\big(x(\delta),x(0) \big) \}\big) \leq P\big(\{x; d_{J_1}^0\big(x(\delta),x(0) \big) \}\big)<\eta.$$

We prove that part (c) of $(iii')$ holds. By the left continuity of elements in $\bD([0,1];\bD)$,
$P\big(\{x; d_{J_1}^0\big(x(1-),x(1-\delta) \big) \geq \e \}\big) \to 0$ as $\delta \to 0$, for any $\e>0$. By \eqref{analogue-13-9}, it follows that
$P\big(\{x; d_{J_1}^0\big(x(1),x(1-) \big) \geq \e \}\big) =0$, for any $\e>0$. Hence,
$P\big(\{x; d_{J_1}^0\big(x(1),x(1-) \big)>0 \}\big) =0$.
The rest of the argument is the same as for part (b). $\Box$

\vspace{3mm}

The previous theorem can also be stated in terms of random elements, as follows.

\begin{theorem}
\label{analogue-th-13-3-distr}
Let $(X_n)_{n\geq 1}$ and $X$ be random elements in $\bD([0,1];\bD)$ defined on the same probability space. Let
$T_X=\{t \in [0,1];P(X(t)=X(t-))=1\}$.  Suppose that:\\
a) $\big(X_n(t_1),\ldots,X_n(t_k) \big) \stackrel{d}{\to} \big(X(t_1),\ldots,X(t_k) \big)$ in $\bD^k$,
for any $t_1,\ldots,t_k \in T_X$;\\
b) $d_{J_1}^0\big(X(1),X(1-\delta)\big) \stackrel{P}{\to}0$ as $\delta \to 0$;\\
c) for any $\e>0$,
\begin{align*}
\left\{
\begin{array}{ll}
& \lim_{\delta \to 0} \limsup_{n\to \infty} P\big(\{w''\big(X_n(t),\delta \big)\geq \e \ \mbox{for some} \ t \in [0,1] \}\big)=0, \\
&  \lim_{\delta \to 0}\limsup_{n \to \infty}P(|X_n(t,\delta)-X_n(t,0)| \geq  \e \ \mbox{for some} \ t \in [0,1])=0,\\
& \lim_{\delta \to 0}\limsup_{n \to \infty}P(|X_n(t,1-)-X_n(t,1-\delta)| \geq  \e \ \mbox{for some} \ t \in [0,1])=0;
\end{array}
\right.
\end{align*}
d) for any $\e>0$,
\begin{equation}
\label{w-second-cond}
\lim_{\delta \to 0}\limsup_{n\to \infty}P(w_{\bD}''(X_n,\delta)\geq \e)=0 \quad \mbox{for all} \ \e>0.
\end{equation}
Then $X_n \stackrel{d}{\to}X$ in $\bD([0,1];\bD)$ equipped with $d_{\bD}$.
\end{theorem}

\begin{remark}
{\rm Hypothesis {\em c)} of Theorem \ref{analogue-th-13-3-distr} may be difficult to verify in practice. In the proof of Theorem 3.14 of \cite{balan-saidani18}, this hypothesis is verified by showing that
\begin{equation}
\label{inf-sup-cond}
\inf_{n_0 \geq 1}\sup_{n\geq n_0}P(\|X_n-X_{n_0}\|_{\bD} \geq \e)=0 \quad \mbox{for all} \ \e>0.
\end{equation}
Since for any $n_0 \geq 1$, the single probability measure $P \circ X_{n_0}^{-1}$ is tight in $\bD([0,1];\bD)$, part $(ii')$ of Theorem \ref{analogue-13-8} gives:
\begin{align*}
\left\{
\begin{array}{ll}
& \lim_{\delta \to 0}  P\big(\{w''\big(X_{n_0}(t),\delta \big)\geq \e \ \mbox{for some} \ t \in [0,1] \}\big)=0, \\
&  \lim_{\delta \to 0}P(|X_{n_0}(t,\delta)-X_{n_0}(t,0)| \geq  \e \ \mbox{for some} \ t \in [0,1])=0,\\
& \lim_{\delta \to 0}P(|X_{n_0}(t,1-)-X_{n_0}(t,1-\delta)| \geq  \e \ \mbox{for some} \ t \in [0,1])=0;
\end{array}
\right.
\end{align*}
Hypothesis {\em c)} then follows from \eqref{inf-sup-cond}, using the following inequalities:
\begin{align*}
w''\big(X_n(t),\delta\big) & \leq w''\big(X_{n_0}(t),\delta\big)+2 \|X_{n}-X_{n_0}\|_{\bD} \\
|X_{n}(t,\delta)-X_{n}(t,0)| & \leq |X_{n_0}(t,\delta)-X_{n_0}(t,0)|+2 \|X_{n}-X_{n_0}\|_{\bD} \\
|X_{n}(t,1-)-X_{n}(t,1-\delta)| & \leq |X_{n_0}(t,1-)-X_{n_0}(t,1-\delta)| +2 \|X_{n}-X_{n_0}\|_{\bD}.
\end{align*}
}
\end{remark}

\section{Criteria for existence and convergence}

In this section, we give a criterion for weak convergence of random elements in $\bD([0,1];\bD)$, and a criterion for the existence of a process with sample paths in $\bD([0,1];\bD)$ based on its finite-dimensional distributions. Both these results rely on some maximal inequalities which are of independent interest.

The first two results are analogue of Theorems 10.3 and 10.4 of \cite{billingsley99}, stated in terms of the Skorohod distance $d_{J_1}^0$.

\begin{theorem}
\label{analogue-th-10-3}
Let $T$ be a Borel set in $[0,1]$ and $\{X(t)\}_{t \in T}$ a collection of random elements in $\bD$
defined on the same probability space $(\Omega,\cF,P)$ such that the map $T \ni t \mapsto X(\omega,t)$ is right-continuous with respect to $J_1$, for any $\omega \in \Omega$. (If $T$ is finite, this imposes no restriction.)  For any $r,s,t \in T$ with $r\leq s \leq t$, let
\begin{equation}
\label{def-m-rst}
m_{rst}^{J_1}=d_{J_1}^0\big(X(r),X(s)\big) \wedge d_{J_1}^0\big(X(s),X(t)\big)
\end{equation}
and
$L_{J_1}(X)=\sup_{r,s,t \in T;\, r\leq s \leq t} m_{rst}^{J_1}$. Suppose that there exist $\alpha>1/2$, $\beta \geq 0$ and a finite measure $\mu$ on $T$ such that for any $\lambda>0$ and for any $r,s,t \in T$ with $r\leq s \leq t$,
\begin{equation}
\label{analogue-10-15}
P(m_{rst}^{J_1} \geq \lambda) \leq \frac{1}{\lambda^{4\beta}} \{\mu(T \cap (r,t])\}^{2\alpha}.
\end{equation}
Then there exists a constant $K$ depending on $\alpha$ and $\beta$ such that for any $\lambda>0$,
\begin{equation}
\label{analogue-10-16}
P(L_{J_1}(X)>\lambda) \leq \frac{K}{\lambda^{4\beta}}\mu^{2\alpha}(T).
\end{equation}
\end{theorem}

\noindent {\bf Proof:} We follow the same idea as in the proof of Theorem 10.3 of \cite{billingsley99}, replacing increments of the form $|X(t)-X(s)|$ by $d_{J_1}^0(X(t),X(s))$.

{\em Case 1.} $T=[0,1]$ and $\mu$ is the Lebesgue measure. Let $D_k=\{i/2^k;0\leq i\leq 2^k\}$. Define
$B_k$ be the maximum of all $m_{t_1 t_2 t_3}^{J_1}$ for all $t_1,t_2,t_3 \in D_k$ with $t_1 \leq t_2 \leq t_3$ and
$A_k$ be the maximum of $m_{t_1 t_2 t_3}^{J_1}$ with $t_1=(i-1)/2^{k}$, $t_2=i/2^{k}$ and $t_3=(i+1)/2^{k}$, for $i=1,\ldots,2^k-1$. It can be proved that $B_k \leq 2(A_1+\ldots+A_k)$ for any $k\geq 1$. Note that $B_k \leq B_{k+1}$ for all $k\geq 1$. We claim that:
\begin{equation}
\label{LJ1=limit}
L_{J_1}(X)=\lim_{k \to \infty}B_k.
\end{equation}
To see this, let $\e>0$ be arbitrary. Let $t_1,t_2,t_3 \in T$ be such that $t_1 \leq t_2 \leq t_3$. For each $k\geq 1$, there exist $t_1^k,t_2^k,t_3^k \in D_k$ with $t_1^k \leq t_2^k \leq t_3^k$ such that $t_i^k \downarrow t_i$ as $k \to \infty$, for $i=1,2,3$. Since $t \mapsto X(t)$ is right-continuous with respect to $J_1$,
$X(t_i^k) \stackrel{J_1}{\to} X(t_i)$ as $k \to \infty$, for $i=1,2,3$. By Lemma \ref{cont-Phi},
$a_k=d_{J_1}^0\big(X(t_1^k),X(t_2^k) \big) \to a=d_{J_1}^0\big(X(t_1),X(t_2) \big)$ as $k \to \infty$ and
$b_k=d_{J_1}^0\big(X(t_2^k),X(t_3^k) \big) \to b=d_{J_1}^0\big(X(t_2),X(t_3) \big)$ as $k \to \infty$.
Hence, there exists $k_{\e}$ such that $a_{k_e} \geq a-\e$ and $b_{k_{\e}} \geq b-\e$. So, $a \wedge b \leq a_{k_e} \wedge b_{k_e} +\e \leq B_{k_e}+\e$. Since $t_1,t_2,t_3$ were arbitrary, we obtain that $L_{J_1}(X) \leq B_{k_e}+\e$.

From \eqref{LJ1=limit}, it follows that $L_{J_1}(X) \leq 2\sum_{k\geq 1}A_k$. From this, we deduce relation \eqref{analogue-10-16} using \eqref{analogue-10-15} to estimate the tail probability of $A_k$ (see page 110 of \cite{billingsley99}).

The other cases follow as in the proof of Theorem 10.3 of \cite{billingsley99}. $\Box$

\vspace{3mm}


\begin{corollary}
\label{analogue-th-10-4}
If condition \eqref{analogue-10-15} of Theorem \ref{analogue-th-10-3} only holds for $t-r<2\delta$, then
$$P(L_{J_1}(X,\delta)>\lambda) \leq \frac{2K}{\lambda^{4\beta}}\mu(T) \sup_{0\leq t \leq 1-2\delta}\mu^{2\alpha-1}\big(T\cap [t,t+2\delta]\big),$$
where $L_{J_1}(X,\delta)$ is the supremum of $m_{rst}^{J_1}$ for all $r,s,t \in T$ with $r\leq s \leq t$ and $t-r<\delta$, and $m_{rst}^{J_1}$ is given by \eqref{def-m-rst}. In particular, if $T=[0,1]$, then $L_{J_1}(X,\delta)=w_{\bD}''(X,\delta)$.
\end{corollary}


The following result gives a criterion for convergence in distribution in the space $\bD([0,1];\bD)$. being the analogue of Theorem 13.5 of \cite{billingsley99}. 

\begin{theorem}
\label{analogue-th-13-5}
Let $(X_n)_{n\geq 1}$ and $X$ be random elements in $\bD([0,1];\bD)$ defined on the same probability space, such that hypotheses a),b),c) of Theorem \ref{analogue-th-13-3-distr} hold. If there exist $\alpha >1/2$, $\beta \geq 0$ and a non-decreasing continuous function $F$ on $[0,1]$ such that for any $r,s,t \in [0,1]$ with $r \leq s \leq t$, for any $\lambda>0$ and for any $n\geq 1$,
$$P\big(d_{J_1}^0(X_n(r),X_n(s)) \wedge d_{J_1}^0(X_n(s),X_n(t)) \geq \lambda \big) \leq \frac{1}{\lambda^{4\beta}}[F(t)-F(r)]^{2\alpha},$$
then $X_n \stackrel{d}{\to} X$ in $\bD([0,1];\bD)$ equipped with $d_{\bD}$.
\end{theorem}

\begin{proof}
We apply Theorem \ref{analogue-th-13-3-distr}. Hypothesis \eqref{w-second-cond} of this theorem is verified using Corollary \ref{analogue-th-10-4} with $T=[0,1]$.
\end{proof}

The goal of the remaining part of this section is to give a criterion for the existence of a process with sample paths in $\bD([0,1];\bD)$. For this, we first need to state a variant of Theorem \ref{analogue-th-10-3} using the uniform norm $\|\cdot\|$ on $\bD$ (instead of the Skorohod distance $d_{J_1}^0$).

 \begin{theorem}
 \label{variant-th-10-3}
Let $T$ be a Borel set in $[0,1]$ and $\{X(t)\}_{t \in T}$ a collection of random elements in $\bD$
defined on the same probability space $(\Omega,\cF,P)$ such that the map $T \ni t \mapsto X(\omega,t)$ is right-continuous with respect to the uniform norm on $\bD$, for any $\omega \in \Omega$. (If $T$ is finite, this imposes no restriction.)  For any $r,s,t \in T$ with $r\leq s \leq t$, let
\begin{equation}
\label{def-m-rst-u}m_{rst}^{u}=\|X(r)-X(s)\| \wedge\|X(s)-X(t)\|
\end{equation} and
$L_{u}(X)=\sup_{r,s,t \in T;\, r\leq s \leq t} m_{rst}^{u}$. Suppose that there exist $\alpha>1/2$, $\beta \geq 0$ and a finite measure $\mu$ on $T$ such that for any $\lambda>0$ and for any $r,s,t \in T$ with $r\leq s \leq t$,
\begin{equation}
\label{analogue-10-15-variant}
P(m_{rst}^{u} \geq \lambda) \leq \frac{1}{\lambda^{4\beta}} \{\mu(T \cap (r,t])\}^{2\alpha}.
\end{equation}
Then there exists a constant $K$ depending on $\alpha$ and $\beta$ such that for any $\lambda>0$,
\begin{equation}
\label{analogue-10-16-variant}
P(L_{u}(X)>\lambda) \leq \frac{K}{\lambda^{4\beta}}\mu^{2\alpha}(T).
\end{equation}
\end{theorem}


\begin{corollary}
\label{analogue-th-10-variant}
If condition \eqref{analogue-10-15-variant} of Theorem \ref{variant-th-10-3} only holds for $t-r<2\delta$, then
$$P(L_{u}(X,\delta)>\lambda) \leq \frac{2K}{\lambda^{4\beta}}\mu(T) \sup_{0\leq t \leq 1-2\delta}\mu^{2\alpha-1}\big(T\cap [t,t+2\delta]\big),$$
where $L_{u}(X,\delta)$ is the supremum of $m_{rst}^{u}$ for all $r,s,t \in T$ with $r\leq s \leq t$ and $t-r<\delta$, and $m_{rst}^{u}$ is given by \eqref{def-m-rst-u}. In particular, if $T=[0,1]$, then $L_{u}(X,\delta)=w_{u}''(X,\delta)$, where
$$w_{u}''(x,\delta)=\sup_{t_1 \leq t \leq t_2,t_2-t_1 \leq \delta}\big(\|x(t)-x(t_1)\|\wedge \|x(t_2)-x(t)\|\big)$$
for any $x \in \bD([0,1];\bD)$.
\end{corollary}

In the particular case when $T$ is a finite set, we obtain the following result, which is of independent interest.

\begin{theorem}
Let $\xi_1, \ldots, \xi_n$ be random elements in $\bD([0,1];\bD)$, $S_k=\xi_1+\ldots+\xi_k$ for $k=1,\ldots,n$, and $S_0=0$. Suppose that there exist $\alpha \geq 1/2$, $\beta>0$ and $u_i\geq 0 ,i=1,\ldots,n$ such that for any $\lambda>0$,
$$P\big(\|S_j-S_i\| \wedge \|S_k-S_j\| \geq \lambda \big) \leq \frac{1}{\lambda^{4\beta}} \Big(\sum_{j=i+1}^{k}u_j \Big)^{2\alpha}.$$
Then there exist a constant $K$ depending on $\alpha$ and $\beta$ such that for any $\lambda>0$,
$$P(M_n \geq \lambda) \leq \frac{K}{\lambda^{2\beta}} \Big(\sum_{i=1}^{n}u_i \Big)^{2\alpha},$$
where $M_n=\max_{0\leq i \leq j \leq k \leq n}\big(\|S_j-S_i\| \wedge \|S_k-S_j\| \big)$.
\end{theorem}

We are now ready to state the criterion for existence of a process with sample paths in $\bD([0,1];\bD)$.

\begin{theorem}
\label{analogue-th-13-6}
Let $\{X(t)\}_{t \in [0,1]}$ be a collection of random elements in $\bD$ defined on the same probability space $(\Omega,\cF,P)$ such that:\\
a) there exist $\alpha >1/2$, $\beta\geq 0$ and a non-decreasing continuous function $F$ on $[0,1]$ such that for any $t_1,t_2,t_3 \in [0,1]$ with $t_1 \leq t_2 \leq t_3$ and for any $\lambda>0$,
$$P\big(\|X(t)-X(t_1)\| \wedge \|X(t_2)-X(t)\| \geq \lambda \big) \leq \frac{1}{\lambda^{4\beta}} [F(t_2)-F(t_1)]^{2\alpha};$$
b) for any $\e>0$,
$$\lim_{\delta \to 0} \limsup_{n\to \infty} P\big(w'\big(X(i/2^{n}), \delta\big)\geq \e \ \mbox{for some} \ 0\leq i\leq 2^{-n} \big)=0;$$
c) for any $t \in [0,1)$ and for any sequence $(t_n)_{n \geq 1}$ in $[0,1]$ with $t_n \to t$ and $t_{n+1} \leq t_n$ for any $n\geq 1$, $d_{J_1}^0\big(X(t_n),X(t)\big) \stackrel{p}{\to}0$.

Then, there exists a collection $\{Y(t)\}_{t \in [0,1]}$ of random elements in $\bD$ defined on the another probability space $(\Omega',\cF',P')$, such that the map $t \mapsto Y(\omega',t)$ is in $\bD([0,1];\bD)$ for any $\omega' \in \Omega'$, and the vectors $(X(t_1),\ldots,X(t_k))$ and $(Y(t_1),\ldots,Y(t_k))$ have the same distribution in $\bD^k$, for any $t_1,\ldots,t_k \in [0,1]$ and for any $k\geq 1$.
\end{theorem}

\noindent {\bf Proof:} We argue as in the proof of Theorem 13.6 of \cite{billingsley99}. We consider two cases.

{\em Case 1.} Suppose that there exists $\delta_0 \in (0,1/2)$ such that for all $h \in (0,\delta_0)$,
\begin{equation}
\label{cond-C}
P(X(0)=X(h))=1 \quad \mbox{and} \quad P(X(1)=X(1-h))=1.
\end{equation}

Let $T=\cup_{n\geq 1}T_n$, where $T_n=\{t_i^n;i=0,1,\ldots,2^n\}$ and $t_i^n=i/2^n$. For any $n\geq 1$, we define \begin{equation}
\label{def-Xn}
X_n(t)=X(t_i^n) \quad \mbox{for all}  \ t \in [t_i^n, t_{i+1}^n)
\end{equation}
for $i=0,1,\ldots,2^n-1$, and $X_n(1)=X(1)$. Note that
\begin{equation}
\label{t-in-Ak}
t \in T_{k} \ \mbox{implies that} \ X_n(t)=X(t) \ \mbox{for all}  \ n \geq k.
\end{equation}

We will prove that
\begin{equation}
\label{Xn-tight}
(P \circ X_n^{-1})_{n\geq 1} \ \mbox{is tight in} \ \bD([0,1];\bD).
\end{equation}
By Prohorov's theorem, it will follow that $(P \circ X_n^{-1})_{n\geq 1}$ is relatively compact in
$\bD([0,1];\bD)$. Hence, there exist a subsequence $(n_k)_{k\geq 1}$ and a probability measure $Q$ on
$(\bD([0,1];\bD), \cD_{\bD})$ such that $P_{n_k} \stackrel{w}{\to} Q$. Let $\{Y(t)\}_{t \in [0,1]}$ be a collection of random elements in $\bD$ with law $Q$, defined on the another probability space $(\Omega',\cF',P')$. For instance, we may take $(\Omega',\cF',P')=\big(\bD([0,1];\bD), \cD_{\bD},Q \big)$ and $Y(t)=\pi_{t}^{\bD}$ for all $t \in [0,1]$. Then  $(X(t_1),\ldots,X(t_k))$ and $(Y(t_1),\ldots,Y(t_k))$ have the same distribution in $\bD^k$, for any $t_1,\ldots,t_k \in A$, and the same thing remains true for arbitrary points $t_1,\ldots,t_k$ in $[0,1]$ due to hypothesis {\em c)} and the right continuity of the sample paths of $\{Y(t)\}_{t \in [0,1]}$ with respect to $J_1$, since for each $t_i \in [0,1),i=1,\ldots,k$ there exists a sequence $(t_i^m)_{m \geq 1} \subset A$ such that $t_i^m \downarrow t_i$ as $m \to \infty$.

It remains to prove \eqref{Xn-tight}. For this, we apply Theorem \ref{analogue-13-8} to $P_n=P \circ X_n^{-1}$. Condition $(ii')$ of this theorem is equivalent to condition $(ii)$ of Theorem \ref{analogue-th-13-2}, which is the same as our hypothesis {\em b)} (using definition \eqref{def-Xn} of $X_n(t)$).

We begin by checking condition $(iii')$ of Theorem \ref{analogue-13-8}. Let $\e>0$ and $\eta>0$ be arbitrary. We prove that there exist $\delta \in (0,1)$ and an integer $n_0 \geq 1$ such that for all $n \geq n_0$,
\begin{equation}
\label{cond-threeparts}
\left\{
\begin{array}{ll}
(a) & P(w_{\bD}''(X_n,\delta) \geq \e ) \leq \eta \\
(b) & P(d_{J_1}^{0}\big(X_n(\delta),X_n(0)\big) \geq \e ) \leq \eta \\
(c) & P(d_{J_1}^{0}\big(X_n(1-),X_n(1-\delta)\big) \geq \e ) \leq \eta.
\end{array}
\right.
\end{equation}
For part b), let $\delta<\delta_0$ and $n\geq 1$ be arbitrary. Choose $k$ such that $k/2^n \leq \delta <(k+1)/2^n$. Then $X_n(\delta)=X(k/2^n)=X(0)$ a.s. and $X_n(0)=X(0)$. Hence $d_{J_1}^{0}\big(X_n(\delta),X_n(0)\big)=0$ a.s.
For part c), let $\delta<\delta_0/2$ and $n\geq n_0$ where $n_0$ is such that $2^{-n_0}\leq \delta_0/2$. Choose $l$ such that $l/2^n \leq 1-\delta <(l+1)/2^n$.
Then $X_n(1-\delta)=X(l/2^n)=X(1)$ a.s. since $1-l/2^n=\delta+(1-\delta-l/2^n)<\delta+1/2^n<\delta_0$.
Since $\delta<\delta_0$ is arbitrary, this also shows that $X_n(1-)=X(1)$ a.s. for any $n\geq n_0$.
Hence, $d_{J_1}^{0}\big(X_n(1-),X_n(1-\delta)\big)=0$ a.s.

To prove part (a) of \eqref{cond-threeparts}, it suffices to show that $P(w_{u}''(X_n,\delta) \geq \e ) \leq \eta$ since $w_{\bD}''(x,\delta)\leq w_{u}''(x,\delta)$ for any $x \in \bD([0,1];\bD)$. This can be proved exactly as on page 144 of \cite{billingsley99}, by applying Corollary
\ref{analogue-th-10-variant} to the discrete-time process $\{Y_n(t)\}_{t \in T_n}$ given by $Y_n(t)=X_n(t)=X(t)$, and the measure $\mu_n$ on $T_n$ given by $\mu_n(\{t_i^n\})=F(t_i^n)-F(t_{i-1}^n)$. The process $Y_n$ satisfies hypothesis \eqref{analogue-10-15-variant} of Corollary \ref{analogue-th-10-variant}, due to our hypothesis {\em a)}.
Note that $w_u''(X_n,\delta) \leq L(Y_n,2\delta)$.

Finally, we prove that condition $(i)$ of Theorem \ref{analogue-th-13-2} holds. Let $\eta>0$ be arbitrary. We will prove that there exist $a>0$ and an integer $n_0 \geq 1$ such that
$$P(\|X_n\|_{\bD} \geq a) \leq 2\eta \quad \mbox{for all} \ n\geq n_0.$$
Let $\e>0$ be arbitrary. Choose $\delta \in (0,1)$ and $n_0 \geq 1$ such that part (a) of \eqref{cond-threeparts} holds. Choose $k\geq 1$ such that $2^{-k}\leq \delta$. We claim that for all $n\geq k$,
\begin{equation}
\label{bound-norm-Xn}
\|X_n\|_{\bD} \leq w_{\bD}''(X_n,\delta)+\max_{i\leq 2^{k}}\|X(i/2^k)\|.
\end{equation}
To see this, note that clearly $\|X_n(1)\| \leq \max_{i\leq 2^{k}}\|X(i/2^k)\|$. Let $t \in [0,1)$ be arbitrary.
Say $i/2^k \leq t <(i+1)/2^k$. We have two situations: $d_{J_1}^0\big(X_n(t),X_n(i/2^k) \big)$ is either smaller or larger than $d_{J_1}^0\big(X_n(t),X_n((i+1)/2^k) \big)$. We consider only the case when it is smaller, the other case being similar. By \eqref{d-equal-norm} and the triangle inequality in $\bD$, we have
$$\|X_n(t)\| \leq d_{J_1}^0\big(X_n(t),X_n(i/2^k) \big)+\|X_n(i/2^k)\| \leq w_{\bD}''(X_n,\delta)+\max_{i\leq 2^{k}}\|X_n(i/2^k)\|.$$
From \eqref{bound-norm-Xn} and part (a) of \eqref{cond-threeparts}, it follows that for all $n\geq n_0$,
\begin{align*}
P(\|X_n\|_{\bD}>a)&\leq P(w_{\bD}''(X_n,\delta)+\max_{i\leq 2^{k}}\|X(i/2^k)\|>a,w_{\bD}''(X_n,\delta) < \e)+P(w_{\bD}''(X_n,\delta)\geq \e)\\
&\leq  P(\max_{i\leq 2^{k}}\|X(i/2^k)\|>a-\e)+P(w_{\bD}''(X_n,\delta)\geq \e)\leq 2\eta,
\end{align*}
for all $a>a_0$ and some $a_0>\e$ large enough, since $\lim_{A \to \infty}P(\max_{i\leq 2^{k}}\|X(i/2^k)\|>A)=0$.

{\em Case 2.} In the absence of condition \eqref{cond-C}, let $\delta_0 \in (0,1/2)$ be arbitrary. For any $t \in [0,1]$, define $\widetilde{X}(t)=X(f(t))$ where
$$f(t)=\left\{
\begin{array}{ll} 0 & \mbox{if $t \in [0,\delta_0)$} \\
(t-\delta_0)/(1-2\delta_0) & \mbox{if $t \in [\delta_0,1-\delta_0]$} \\
1 & \mbox{if $t \in (1-\delta_0,1]$}
\end{array} \right.$$
Since the map $\phi:[\delta_0,1-\delta_0] \to [0,1]$ given by $\phi(t)=(t-\delta_0)/(1-2\delta_0)$ is a bijection, $X(s)=\widetilde{X}(\delta_0+(1-2\delta_0)s)$ for all $s \in [0,1]$. The process $\widetilde{X}$ satisfies hypotheses {\em a)}, {\em b)}, {\em c)} of the theorem, and also condition \eqref{cond-C}. Therefore, by {\em Case 1}, there exists a collection $\{\widetilde{Y}(t)\}_{t \in [0,1]}$ of random elements in $\bD$ defined on the another probability space $(\Omega',\cF',P')$, such that the map $t \mapsto \widetilde{Y}(\omega',t)$ is in $\bD([0,1];\bD)$ for any $\omega' \in \Omega'$, and the vectors $(\widetilde{X}(t_1),\ldots,\widetilde{X}(t_k))$ and $(\widetilde{Y}(t_1),\ldots,\widetilde{Y}(t_k))$ have the same distribution in $\bD^k$, for any $t_1,\ldots,t_k \in [0,1]$ and for any $k\geq 1$. We define $Y(s)=\widetilde{Y}(\delta_0+(1-2\delta_0)s)$ for all $s \in [0,1]$.
$\Box$

\end{document}